\documentclass[10pt]{article}

\usepackage[letterpaper,margin=0.88in]{geometry}
\usepackage[T1]{fontenc}
\usepackage{lmodern}
\usepackage{amsmath,amssymb,amsthm,mathtools}
\usepackage{booktabs}
\usepackage{graphicx}
\usepackage{xcolor}
\usepackage{enumitem}
\usepackage{hyperref}
\usepackage[nameinlink,noabbrev]{cleveref}

\hypersetup{
  colorlinks=true,
  linkcolor=blue!55!black,
  citecolor=blue!55!black,
  urlcolor=blue!55!black,
  pdftitle={Farthest-cell triplet entropy: high-dimensional shell limits and hyperbolic curvature amplification},
  pdfauthor={Chongkun Deng},
  pdfsubject={High-dimensional probability, comparison-based entropy, and hyperbolic geometry},
  pdfkeywords={conditional entropy, farthest-site cells, high-dimensional geometry, hyperbolic curvature, ordinal comparisons}
}
\graphicspath{{paper_outputs/}}
\setlist{nosep,leftmargin=*}
\newtheorem{theorem}{Theorem}
\newtheorem{proposition}{Proposition}
\newtheorem{lemma}{Lemma}

\theoremstyle{definition}
\newtheorem{definition}{Definition}

\newcommand{\Htri}{H_3}
\newcommand{\Hinf}{H_{\infty}}
\newcommand{\metric}{\mathsf d}
\newcommand{\E}{\mathbb{E}}
\newcommand{\Pp}{\mathbb{P}}

\newcommand{\one}{\mathbf{1}}

\newcommand{\doi}[1]{\href{https://doi.org/#1}{doi:#1}}

\title{\textbf{Farthest-cell triplet entropy: high-dimensional shell limits and hyperbolic curvature amplification}}
\author{%
  Chongkun Deng%
  \thanks{Correspondence: \href{mailto:cd3411@columbia.edu}{cd3411@columbia.edu}. ORCID: \href{https://orcid.org/0009-0005-0808-9278}{0009-0005-0808-9278}.}%
  \thanks{AI-use disclosure: ChatGPT and Gemini assisted with language and code. The author verified the final manuscript and computations.}\\[-0.2em]
  \small Independent researcher (no institutional affiliation)
}
\date{Preprint, September 2026}

\begin{document}
\maketitle

\begin{abstract}
We introduce farthest-cell triplet entropy, the conditional Shannon entropy of the farthest-prototype label given three random prototypes. For independent queries and prototypes, its estimator records only the farthest label, not coordinates or numerical distances. The statistic is bounded by $\log 3$, is invariant under common strictly increasing transformations of the dissimilarities, and has an exact mutual-information interpretation. In high-dimensional isotropic radial models $X_d=R_dU_d$, the Euclidean ordering reduces to scores $\lambda_d\xi_{i,d}-Z_i$, where $\lambda_d=\sqrt d\,\operatorname{sd}(R_d)/\E R_d$ and the $Z_i$ are independent standard Gaussian variables. This gives angular-dominated, intermediate, and radial-dominated entropy limits $\log 3$, $\Hinf(\lambda;F)$, and $0$. In hyperbolic space of curvature $-\kappa_d^2$, the same master curve appears at $\lambda_{d,\mathbb H}=\sqrt d\,\tau_d A(s_d)$, where $\tau_d=\operatorname{sd}(R_d)/\E R_d$, $s_d=\kappa_d\E R_d$, and $A(s)=s\coth s$. With a calibrated radial law and $\tau_d$, and a monotone operating interval, entropy inversion identifies the scale-invariant target $s_d^2$; absolute curvature requires an external length unit. CPU simulations give Euclidean and hyperbolic master-curve RMSEs of $0.0164$ and $0.0209$. Inversion from observed synthetic latent coordinates has a median relative error in $\kappa$ of $6.6\%$, while angular anisotropy increases this error to $68.9\%$. Thus the entropy statistic is comparison-based, whereas curvature recovery remains model-calibrated, is not graph-only, and is not robust to anisotropy.
\end{abstract}

\section{Introduction}

Let a query and three prototypes be sampled independently from a population on a metric space. Assign the query the label of its farthest prototype. Conditional on the prototype triple, the three label probabilities are the population masses of its random farthest-site cells. We define farthest-cell triplet entropy as the Shannon entropy of this conditional label distribution, averaged over the prototype triple. The statistic measures how predictable the farthest label becomes after the prototypes are observed: it is near $\log 3$ when their cells are balanced and near zero when one cell dominates.

Given sampled queries and prototypes, estimation requires only the farthest-of-three label, not the three distance values. This is the precise sense in which the statistic is comparison-based. Curvature calibration requires additional information: the model dimension, relative shell thickness, radial law, and validity of the isotropic shell model. These quantities may be supplied by coordinates or by an ordinal embedding under separate recovery assumptions. Because ordinal embeddings identify scale only up to a convention, their global scale is a coordinate gauge rather than metric information created by the comparisons \cite{kleindessner2014,ariascastro2017,suzuki2019}.

To state the main mechanism precisely, consider for each dimension $d$ the isotropic Euclidean radial model
\begin{equation}
  X_d=R_dU_d\in\mathbb R^d,\qquad
  U_d\sim\operatorname{Unif}(S^{d-1}),\qquad R_d\perp U_d,
  \label{eq:radial-model-intro}
\end{equation}
where $R_d>0$ is the random radius. Let $\rho_d=\E R_d$, $\epsilon_d=\operatorname{sd}(R_d)$, and $\xi_d=(R_d-\rho_d)/\epsilon_d$. A prototype-radius fluctuation has scale $\epsilon_d$, whereas the angular fluctuation has scale $\rho_d/\sqrt d$. Their ratio
\begin{equation}
  \lambda_d=\sqrt d\,\frac{\epsilon_d}{\rho_d}
  \label{eq:lambda-flat}
\end{equation}
governs a one-parameter Gaussian noisy-ranking problem. In hyperbolic space, the law of cosines changes this ratio by
\begin{equation}
  A(s)=s\coth s,\qquad s=\kappa_d\rho_d.
  \label{eq:amplification}
\end{equation}
Curvature therefore changes the effective radial-to-angular signal without creating a new entropy curve. Writing $\tau_d=\epsilon_d/\rho_d$ and $s_d=\kappa_d\rho_d$ gives
\begin{equation}
  \lambda_{d,\mathbb H}=\sqrt d\,\tau_d A(s_d).
  \label{eq:lambda-dimensionless-intro}
\end{equation}
Both $\tau_d$ and $s_d$ are unchanged by a global change of length unit. Consequently, once $\tau_d$ and the radial law are separately calibrated and the master curve is verified to be monotone on the operating interval, entropy inversion targets the normalized curvature $s_d^2=\kappa_d^2\rho_d^2$, not absolute curvature in an unspecified unit.

Our contributions are:
\begin{enumerate}
  \item a conditional-entropy definition, an exact mutual-information identity, a consistent label-based estimator, a universal ceiling, and invariance under common strictly increasing transformations of the dissimilarities;
  \item a Euclidean score reduction establishing a three-regime asymptotic law, indexed by $\lambda_d$, under relative shell thinness;
  \item a hyperbolic score reduction that retains the random query radius and shows that curvature multiplies the Euclidean parameter by $A(\kappa_d\rho_d)$ under an explicit curvature-scale thinness condition;
  \item a dimensionless inverse formulation for $\kappa_d^2\rho_d^2$, together with similarity-gauge invariance and explicit radial-law, isotropy, and monotonicity assumptions;
  \item a reproducible finite-dimensional study with residual analysis, a controlled latent-network calibration experiment, and an anisotropy stress test.
\end{enumerate}

The inferential scope is model-specific. Absolute curvature recovery requires an external length unit; margin-normalized coordinates support only curvature in the corresponding embedding units. The isotropic master curve does not become valid merely because the data are represented ordinally. The anisotropy stress test provides a negative control, and no anisotropy correction is claimed. The synthetic network experiment observes latent coordinates and therefore evaluates latent-space calibration rather than graph-only inference.

\section{Related work}

\paragraph{Entropy and random partitions.}
Shannon entropy measures categorical balance \cite{shannon1948}. Plug-in entropy estimates have a leading finite-sample bias, with a classical correction due to Miller and later systematic analysis by Paninski \cite{miller1955,paninski2003}. Voronoi diagrams and farthest-site cells are standard computational-geometric objects \cite{aurenhammer1991}. Existing ``Voronoi entropy'' commonly summarizes the distribution of polygon types or cell morphology in a spatial tessellation \cite{bormashenko2024}. Our object instead takes the entropy of the \emph{probability masses} of three random farthest-site cells and then averages over the random prototype triple. To our knowledge, its high-dimensional shell transition has not been analyzed previously.

\paragraph{High-dimensional spherical geometry.}
The Gaussian limit of finitely many coordinates of a random point on a high-dimensional sphere is a classical Poincar\'e-type phenomenon; the conditional version used here follows after the random prototype Gram matrix converges to the identity \cite{diaconisfreedman1987}. Thin-shell asymptotics also appear in the study of rotationally invariant random simplices \cite{heiny2022}. Our use is different: the spherical Gaussian fluctuation competes with a radial perturbation in a conditional farthest-cell mass.

\paragraph{Ordinal embedding and scale.}
Ordinal embedding reconstructs configurations from comparisons such as $d_{ij}<d_{k\ell}$. Under dense regular Euclidean comparisons, uniqueness is only up to translation, rotation, reflection, and a common scale \cite{kleindessner2014}; consistency and local-comparison variants require additional sampling conditions \cite{ariascastro2017}. The hyperbolic ordinal-embedding method of Suzuki et al. uses fixed curvature and a loss-margin convention to represent hierarchical data \cite{suzuki2019}. Our entropy statistic can be computed before embedding from farthest-of-three labels. Curvature inversion, however, also requires the dimension, relative shell thickness, and radial law. An embedding can supply those calibration quantities only under separate recovery assumptions. Accordingly, we report normalized curvature and treat any embedding margin as a chosen gauge, not as observed physical scale.

\paragraph{Hyperbolic representations and networks.}
Negative curvature supports parsimonious representations of hierarchies and heterogeneous networks \cite{krioukov2010,papadopoulos2012,nickelkiela2017}. Modern representation methods learn constant or mixed curvatures \cite{gu2019,skopek2020,bachmann2020}. Recent latent-network work directly estimates curvature by likelihood and studies identifiability and rates \cite{lixuzhu2026}. Our statistic is neither a likelihood nor an embedding algorithm. It is a rank-based summary for data already associated with a specified radial metric model, and it cannot resolve the global distance-scale ambiguity or identify the calibration model by itself.

\section{Farthest-cell triplet entropy}

Let $(\mathcal X,\metric,P)$ be a metric probability space. Draw a query $X\sim P$ and a prototype triple $V=(V_1,V_2,V_3)\sim P^3$ independently. For fixed $v=(v_1,v_2,v_3)$, let
\begin{equation}
  L_v(X)\in\arg\max_{1\le i\le3}\metric(X,v_i)
\end{equation}
be the farthest-prototype label, with uniform randomization among exact ties. Define
\begin{equation}
  q_i(v)=\Pp\{L_V(X)=i\mid V=v\},\qquad
  h(q)=-\sum_{i=1}^3 q_i\log q_i.
\end{equation}

\begin{definition}[Farthest-cell triplet entropy]
The farthest-cell triplet entropy is
\begin{equation}
  \Htri(P,\metric)
  =H\{L_V(X)\mid V\}
  =\E_V h\{q(V)\}.
  \label{eq:def-H3}
\end{equation}
\end{definition}

\begin{proposition}[Information identity, ceiling, and comparison invariance]
\label{prop:ordinal}
For every metric probability space,
\begin{equation}
  \log 3-\Htri(P,\metric)
  =I\{L_V(X);V\},
  \label{eq:general-mi}
\end{equation}
and hence
\begin{equation}
  0\le \Htri(P,\metric)\le\log 3.
\end{equation}
Equality in the upper bound holds if and only if $q(V)=(1/3,1/3,1/3)$ almost surely. Moreover, the value is unchanged when every dissimilarity is transformed by a common strictly increasing function; in particular, it is invariant to global metric rescaling.
\end{proposition}

\begin{proof}
Exchangeability of $(V_1,V_2,V_3)$ gives
$\Pp\{L_V(X)=i\}=1/3$ for each $i$, so
$H\{L_V(X)\}=\log 3$. By \cref{eq:def-H3},
$\Htri(P,\metric)=H\{L_V(X)\mid V\}$; subtracting conditional entropy from marginal entropy proves \cref{eq:general-mi}. Nonnegativity of entropy gives the lower bound, and the upper bound also follows from the information identity. Equivalently, for every $v$, $0\le h\{q(v)\}\le\log 3$, with equality at the upper end exactly when $q(v)$ is uniform; averaging proves the equality characterization. Finally, a common strictly increasing transformation preserves every farthest-label ordering and every tie, hence preserves $q(v)$ and $\Htri$. The transformed dissimilarity need not satisfy the triangle inequality for this last comparison statement because only the induced labels are used.
\end{proof}

The conditioning in \cref{eq:def-H3} is essential. The unconditional label is always uniform, whereas $\Htri$ measures how much observing the random prototypes makes that label predictable. Thus $\log 3-\Htri$ is the information about the prototype configuration carried by one farthest label.

\subsection{Comparison-only estimation and Monte Carlo uncertainty}

Let an oracle return the randomized label $L_v(x)$ from the three comparisons among $\metric(x,v_1),\metric(x,v_2),\metric(x,v_3)$. For prototype triples $V^{(1)},\ldots,V^{(M)}$ sampled independently from $P^3$ and, for each triple, independent witnesses $X_{\ell1},\ldots,X_{\ell N}$ sampled from $P$, define
\begin{equation}
 \widehat q_{\ell i}=\frac1N\sum_{m=1}^N
 \one\{L_{V^{(\ell)}}(X_{\ell m})=i\},\qquad
 \widehat H_{M,N}=\frac1M\sum_{\ell=1}^M h(\widehat q_{\ell}).
 \label{eq:comparison-estimator}
\end{equation}

\begin{proposition}[Consistency from comparisons]
\label{prop:comparison-consistency}
If the prototype triples and witnesses in \cref{eq:comparison-estimator} are sampled independently from $P^3$ and $P$, respectively, then
\begin{equation}
 \widehat H_{M,N}\longrightarrow \Htri(P,\metric)
 \quad\text{in probability whenever }M,N\to\infty.
\end{equation}
The same result holds for the corrected estimator in \cref{eq:miller}.
\end{proposition}

\begin{proof}
Conditionally on $V^{(\ell)}=v$, the label counts are multinomial with category probabilities $q(v)$. Hence
$\widehat q_\ell\to q(V^{(\ell)})$ in probability as $N\to\infty$. Since entropy is continuous and bounded by $\log 3$ on the probability simplex,
\begin{equation}
  a_N:=\E\left|h(\widehat q_\ell)-h\{q(V^{(\ell)})\}\right|
  \longrightarrow0.
  \label{eq:witness-L1}
\end{equation}
Write
\begin{align}
  \widehat H_{M,N}-\Htri(P,\metric)
  &=A_{M,N}+B_M,\label{eq:consistency-decomposition}\\
  A_{M,N}
  &=\frac1M\sum_{\ell=1}^M
    \left[h(\widehat q_\ell)-h\{q(V^{(\ell)})\}\right],\\
  B_M
  &=\frac1M\sum_{\ell=1}^M
    \left[h\{q(V^{(\ell)})\}-\Htri(P,\metric)\right].
\end{align}
For every $\varepsilon>0$, Markov's inequality and \cref{eq:witness-L1} give
\begin{equation}
  \Pp\{|A_{M,N}|>\varepsilon/2\}
  \le \frac{2\E|A_{M,N}|}{\varepsilon}
  \le \frac{2a_N}{\varepsilon}.
  \label{eq:markov-witness}
\end{equation}
The summands defining $B_M$ are independent and bounded in absolute value by $\log 3$, so
\begin{equation}
  \operatorname{Var}(B_M)
  \le\frac{(\log 3)^2}{M}.
\end{equation}
Chebyshev's inequality therefore yields
\begin{equation}
  \Pp\{|B_M|>\varepsilon/2\}
  \le\frac{4(\log 3)^2}{M\varepsilon^2}.
  \label{eq:chebyshev-prototypes}
\end{equation}
Combining \cref{eq:consistency-decomposition,eq:markov-witness,eq:chebyshev-prototypes} with the union bound gives
\begin{equation}
  \Pp\left\{
    |\widehat H_{M,N}-\Htri(P,\metric)|>\varepsilon
  \right\}
  \le
  \frac{2a_N}{\varepsilon}
  +\frac{4(\log 3)^2}{M\varepsilon^2}
  \longrightarrow0
\end{equation}
whenever $M,N\to\infty$. Finally, the correction in \cref{eq:miller} changes each plug-in entropy by at most $1/N$, so it does not change the limit.
\end{proof}

The term comparison-only refers to the geometric observations used after sampling: the construction assumes independent access to the target population $P$, but records only farthest labels. If queries or prototypes are drawn with different weights, the estimator targets the corresponding weighted population functional. No claim is made that the labels alone identify the sampling law or the shell calibration parameters.

The sampling law $P$ and dissimilarity $\metric$ are treated as fixed.
The remaining question is how to quantify the Monte Carlo error caused
by using finitely many prototype triples and witness labels. For each
prototype triple $V^{(\ell)}$, $\ell=1,\ldots,M$, define
\[
  \widehat h_\ell
  =
  h(\widehat q_\ell)
  =
  -\sum_{i=1}^3
  \widehat q_{\ell i}\log \widehat q_{\ell i},
\]
with the convention $0\log 0=0$. Thus $N$ controls the approximation
of the conditional cell probabilities $q(V^{(\ell)})$, whereas $M$
controls the approximation of the outer expectation over prototype
triples.

Because the plug-in entropy has a downward finite-$N$ bias, let
\[
  \widehat s_\ell
  =
  \#\{i:\widehat q_{\ell i}>0\}
\]
and use the adaptive Miller--Madow correction \cite{miller1955,paninski2003}
\begin{equation}
  \widehat h_{\ell,\mathrm{MM}}
  =
  \min\left\{
    \widehat h_\ell+\frac{\widehat s_\ell-1}{2N},
    \log 3
  \right\},
  \qquad
  \widehat H_{M,N}^{\mathrm{MM}}
  =
  \frac{1}{M}\sum_{\ell=1}^M
  \widehat h_{\ell,\mathrm{MM}} .
  \label{eq:miller}
\end{equation}
When all three labels are observed, the correction is $1/N$. Near the
boundary of the probability simplex, it remains a leading-order bias
correction rather than an exactly unbiased estimator.

The Monte Carlo standard error reported in the experiments is
\begin{equation}
  \widehat{\operatorname{se}}_{\mathrm{MC}}
  =
  \left[
    \frac{1}{M(M-1)}
    \sum_{\ell=1}^M
    \left(
      \widehat h_{\ell,\mathrm{MM}}
      -
      \widehat H_{M,N}^{\mathrm{MM}}
    \right)^2
  \right]^{1/2}.
  \label{eq:monte-carlo-se}
\end{equation}
It measures the sampling variability of the Monte Carlo average across
independent prototype triples, including the variability introduced by
the finite witness samples. It does not account for residual finite-$N$
bias, uncertainty from observing only a finite population sample, or
uncertainty introduced by estimating an embedding, metric scale, or
sampling weights.

The procedure requires $MN$ farthest-label observations. When the
labels are computed directly from coordinates in $\mathbb R^d$, its
arithmetic cost is $O(MNd)$, and witnesses may be processed in blocks.

\subsection{Common limiting ranking model and master curve}

The ranking model introduced below is not an auxiliary assumption. It is
the common leading-order form of the Euclidean and hyperbolic
farthest-distance comparisons derived in the following sections.

\paragraph{Euclidean preview.}
Let the query be $X_{0,d}=R_{0,d}U_{0,d}$ and let
$X_{i,d}=R_{i,d}U_{i,d}$, $i=1,2,3$ be prototypes. Their squared
Euclidean distances satisfy
\begin{equation}
  \|X_{0,d}-X_{i,d}\|^2
  =
  R_{0,d}^2+R_{i,d}^2
  -2R_{0,d}R_{i,d}
  \langle U_{0,d},U_{i,d}\rangle.
  \label{eq:euclidean-preview-distance}
\end{equation}
Write
\[
  R_{i,d}=\rho_d+\epsilon_d\xi_{i,d},
  \qquad
  \lambda_d=\sqrt d\,\frac{\epsilon_d}{\rho_d},
  \qquad
  Z_{i,d}
  =
  \sqrt d\,\langle U_{0,d},U_{i,d}\rangle.
\]
After subtracting terms common to the three prototypes and normalizing
by the angular scale $2\rho_d^2/\sqrt d$, the leading score is
\begin{equation}
  S_{i,d}^{\mathrm E}
  \approx
  \lambda_d\xi_{i,d}-Z_{i,d}.
  \label{eq:euclidean-ranking-preview}
\end{equation}
The first term records the prototype's radial advantage, whereas the
second records its angular relation to the query. In high dimension,
the finite collection $(Z_{1,d},Z_{2,d},Z_{3,d})$ is asymptotically
standard Gaussian. The Euclidean farthest-prototype problem therefore
becomes a Gaussian noisy-ranking problem. The score reduction in
\cref{eq:euc-score} makes this approximation precise and controls its
remainder.

\paragraph{Hyperbolic preview.}
In hyperbolic space of curvature $-\kappa_d^2$, the distance ordering is
the ordering of
\begin{equation}
  C_{i,d}
  =
  \cosh(\kappa_dR_{0,d})\cosh(\kappa_dR_{i,d})
  -
  \sinh(\kappa_dR_{0,d})\sinh(\kappa_dR_{i,d})
  \langle U_{0,d},U_{i,d}\rangle.
  \label{eq:hyperbolic-ranking-preview}
\end{equation}
Let $s_d=\kappa_d\rho_d$. A first-order expansion in the prototype
radius gives the normalized leading score
\begin{equation}
  S_{i,d}^{\mathbb H}
  \approx
  \sqrt d\,\kappa_d\epsilon_d
  \coth(\kappa_dR_{0,d})\xi_{i,d}
  -
  Z_{i,d}.
  \label{eq:hyperbolic-preview-query-radius}
\end{equation}
Under the curvature-scale thinness condition introduced in
\cref{eq:hyper-thin},
\[
  \coth(\kappa_dR_{0,d})
  \approx
  \coth(\kappa_d\rho_d).
\]
Consequently,
\begin{align}
  S_{i,d}^{\mathbb H}
  &\approx
  \lambda_{d,\mathbb H}\xi_{i,d}-Z_{i,d},\\
  \lambda_{d,\mathbb H}
  &=
  \sqrt d\,\kappa_d\epsilon_d
  \coth(\kappa_d\rho_d)\\
  &=
  \lambda_d A(s_d),
  \qquad
  A(s)=s\coth s.
  \label{eq:hyperbolic-ranking-preview-reduced}
\end{align}
Thus curvature changes the radial-to-angular signal ratio but does not
produce a different limiting ranking model. The reduction in
\cref{eq:hyper-reduction} supplies the rigorous remainder control.

\paragraph{Common limiting experiment.}
Let $F$ be the limiting distribution of the standardized radial
fluctuation, and let
\[
  W_1,W_2,W_3\stackrel{\mathrm{iid}}{\sim}F,
  \qquad
  Z_1,Z_2,Z_3\stackrel{\mathrm{iid}}{\sim}N(0,1),
\]
with the two triples independent. Motivated by
\cref{eq:euclidean-ranking-preview,eq:hyperbolic-ranking-preview-reduced},
define
\begin{equation}
  Y_i(\lambda)=\lambda W_i-Z_i,
  \qquad
  L_\lambda
  =
  \arg\max_{1\leq i\leq3}Y_i(\lambda).
  \label{eq:limiting-ranking-model}
\end{equation}
Conditionally on $(W_1,W_2,W_3)$, this is a three-alternative Gaussian
random-utility, or multinomial-probit, model \cite{hausmanwise1978}.
Our contribution is its derivation from high-dimensional distance
comparisons and the resulting entropy functional.
Here $\lambda$ is the limiting radial-to-angular signal ratio:
$\lambda=\lim\lambda_d$ in the Euclidean model and
$\lambda=\lim\lambda_{d,\mathbb H}$ in the hyperbolic model.

For a deterministic radial triple $w=(w_1,w_2,w_3)$, the conditional
probability that prototype $i$ wins is
\begin{align}
  q_i^\lambda(w)
  &=
  \mathbb P\!\left(
    L_\lambda=i\mid W_1=w_1,W_2=w_2,W_3=w_3
  \right)\\
  &=
  \int_{-\infty}^{\infty}
  \phi(z)
  \prod_{j\neq i}
  \left[
    1-\Phi\{z+\lambda(w_j-w_i)\}
  \right]
  \,\mathrm dz.
  \label{eq:q-integral}
\end{align}
Indeed, conditional on $Z_i=z$, prototype $i$ wins exactly when
\[
  Z_j\geq z+\lambda(w_j-w_i)
  \qquad\text{for every }j\neq i.
\]

The master entropy curve associated with the radial law $F$ is
\begin{equation}
  H_\infty(\lambda;F)
  =
  \mathbb E
  \left[
    h\{q^\lambda(W_1,W_2,W_3)\}
  \right].
  \label{eq:Hinf}
\end{equation}
This is the common entropy limit produced by both geometries. The
effective parameter $\lambda$ determines which of three regimes occurs:
angular domination at $\lambda=0$, competition at finite positive
$\lambda$, and radial domination as $\lambda\to\infty$.

Exchangeability gives $\mathbb P(L_\lambda=i)=1/3$, so
\begin{equation}
  \log 3-H_\infty(\lambda;F)
  =
  I(L_\lambda;W_1,W_2,W_3).
  \label{eq:mi}
\end{equation}
Thus the entropy defect measures how much the farthest label reveals
about the prototype radial fluctuations. This identity does not imply
that $H_\infty(\lambda;F)$ is monotone for every radial law; inversion
still requires a verified monotone operating interval.

\paragraph{Remark.}
At the level of the limiting ranking experiment, the construction extends
directly to any fixed number \(m>3\) of prototypes and to independent radial
and angular fluctuations after replacing the Gaussian terms \(Z_i\) by the
appropriate limiting angular variables; dependence among angular scores or
radial--angular dependence requires a multivariate, model-specific theory.

\section{Euclidean shell law}

\subsection{Model and limiting notation}

Let
\[
  X_d=R_dU_d\in\mathbb R^d,
  \qquad
  U_d\sim\operatorname{Unif}(S^{d-1}),
  \qquad
  R_d\perp U_d,
\]
and write
\begin{equation}
  R_d
  =
  \rho_d+\epsilon_d\xi_d
  =
  \rho_d(1+\tau_d\xi_d),
  \qquad
  \tau_d=\frac{\epsilon_d}{\rho_d},
  \qquad
  \lambda_d=\sqrt d\,\tau_d.
  \label{eq:euclidean-radial-parameters}
\end{equation}
Assume that $R_d>0$ almost surely,
\[
  \tau_d\longrightarrow0,
  \qquad
  \xi_d\Rightarrow\Xi,
\]
and let $F=\mathcal L(\Xi)$.  The condition $\tau_d\to0$ is the
relative-thinness assumption \cite{heiny2022}.  It permits the effective
radial-to-angular ratio $\lambda_d$ to converge to zero, to a finite
positive constant, or to infinity, provided that
$\lambda_d=o(\sqrt d)$.  For example,
$\tau_d=d^{-\alpha}$ with $0<\alpha<1/2$ gives
$\lambda_d\to\infty$.

Let $X_{0,d}$ be an independent query and let
\[
  V_d=(X_{1,d},X_{2,d},X_{3,d})
\]
be the prototype triple.  Define the finite-dimensional conditional
cell-probability vector
\begin{equation}
  Q_d=(Q_{d1},Q_{d2},Q_{d3}),
  \qquad
  Q_{di}
  =
  \mathbb P\!\left(
    L_{V_d}(X_{0,d})=i\mid V_d
  \right).
  \label{eq:euclidean-conditional-q}
\end{equation}
Thus
\begin{equation}
  H_3(P_d,d_{\mathrm{Euc}})
  =
  \mathbb E\{h(Q_d)\}.
  \label{eq:euclidean-H3-Q}
\end{equation}

For clarity, we also recall the corresponding limiting notation.
Let
\[
  \Xi_1,\Xi_2,\Xi_3\stackrel{\mathrm{iid}}{\sim}F,
  \qquad
  Z_1,Z_2,Z_3\stackrel{\mathrm{iid}}{\sim}N(0,1),
\]
with the two triples independent.  For a deterministic radial triple
$w=(w_1,w_2,w_3)$ and $\lambda\geq0$, define
\begin{equation}
  q_i^\lambda(w)
  =
  \mathbb P\!\left(
    \lambda w_i-Z_i
    =
    \max_{1\leq j\leq3}\{\lambda w_j-Z_j\}
  \right),
  \qquad i=1,2,3,
  \label{eq:euclidean-q-lambda}
\end{equation}
and set
\[
  q^\lambda(w)
  =
  \bigl(q_1^\lambda(w),q_2^\lambda(w),q_3^\lambda(w)\bigr).
\]
For fixed $w$, this is a deterministic probability vector obtained by
averaging over the Gaussian variables $Z_1,Z_2,Z_3$.  Evaluating it at
the random radial triple gives the random vector
\[
  q^\lambda(\Xi_1,\Xi_2,\Xi_3).
\]
The master entropy is
\begin{equation}
  H_\infty(\lambda;F)
  =
  \mathbb E
  \left[
    h\{q^\lambda(\Xi_1,\Xi_2,\Xi_3)\}
  \right],
  \label{eq:euclidean-H-infinity}
\end{equation}
where the remaining expectation is over
$(\Xi_1,\Xi_2,\Xi_3)$.

\subsection{Score reduction}

\begin{lemma}[Euclidean score reduction]
\label{lem:euclidean-score-reduction}
For the independent query $(R_{0,d},U_{0,d})$ and prototypes
$(R_{i,d},U_{i,d})$, $i=1,2,3$, the farthest-distance ordering is the
ordering of
\begin{equation}
  S_{i,d}
  =
  \lambda_d\xi_{i,d}-Z_{i,d}+r_{i,d},
  \qquad
  Z_{i,d}
  =
  \sqrt d\,\langle U_{0,d},U_{i,d}\rangle.
  \label{eq:euc-score}
\end{equation}
Conditionally on the prototypes,
\[
  (Z_{1,d},Z_{2,d},Z_{3,d})
  \Rightarrow N(0,I_3)
\]
in probability \cite{diaconisfreedman1987}.  Moreover,
\[
  \max_{1\leq i\leq3}|r_{i,d}|
  \xrightarrow{p}0
\]
when $\lambda_d=O(1)$, whereas
\[
  \frac{1}{\lambda_d}
  \max_{1\leq i\leq3}|r_{i,d}|
  \xrightarrow{p}0
\]
when $\lambda_d\to\infty$.
\end{lemma}

\begin{proof}
For each prototype,
\[
  \|X_{0,d}-X_{i,d}\|^2
  =
  R_{0,d}^2+R_{i,d}^2
  -
  2R_{0,d}R_{i,d}
  \langle U_{0,d},U_{i,d}\rangle.
\]
Remove the query term $R_{0,d}^2$ and the common prototype term
$\rho_d^2$, and divide by the positive angular scale
$2\rho_d^2/\sqrt d$.  The resulting exact score is
\begin{equation}
  \lambda_d\xi_{i,d}
  +
  \frac{\lambda_d\tau_d}{2}\xi_{i,d}^2
  -
  (1+\tau_d\xi_{0,d})
  (1+\tau_d\xi_{i,d})Z_{i,d}.
  \label{eq:euc-exact}
\end{equation}

Because $\xi_d\Rightarrow\Xi$, the variables
$\xi_{0,d},\xi_{1,d},\xi_{2,d},\xi_{3,d}$ are tight.  Together with
$\tau_d\to0$, this shows that the difference between
\eqref{eq:euc-exact} and
$\lambda_d\xi_{i,d}-Z_{i,d}$ is $o_p(1)$ when
$\lambda_d=O(1)$ and $o_p(\lambda_d)$ when
$\lambda_d\to\infty$.

Conditionally on the prototype directions, the covariance matrix of
$(Z_{1,d},Z_{2,d},Z_{3,d})$ is their Gram matrix,
\[
  \Gamma_d
  =
  \bigl(\langle U_{i,d},U_{j,d}\rangle\bigr)_{i,j=1}^3.
\]
The off-diagonal entries converge to zero in probability, so
$\Gamma_d\to I_3$ in probability. The spherical Poincar\'e lemma then gives
the stated conditional Gaussian limit.
\end{proof}

\subsection{Entropy trichotomy}

\begin{theorem}[Euclidean shell trichotomy]
\label{thm:euclidean-shell-trichotomy}
Under the preceding assumptions:

\begin{enumerate}
  \item if $\lambda_d\to0$, then
  \[
    H_3(P_d,d_{\mathrm{Euc}})\longrightarrow\log 3;
  \]

  \item if $\lambda_d\to\lambda\in(0,\infty)$, then
  \[
    H_3(P_d,d_{\mathrm{Euc}})
    \longrightarrow
    H_\infty(\lambda;F);
  \]

  \item if $\lambda_d\to\infty$ and $F$ is continuous, then
  \[
    H_3(P_d,d_{\mathrm{Euc}})\longrightarrow0.
  \]
\end{enumerate}
\end{theorem}

\begin{proof}
We consider the three regimes separately.

\begin{enumerate}
\item \textbf{Case $\lambda_d\to0$.}

By \cref{lem:euclidean-score-reduction}, conditionally on the
prototypes, the limiting score vector is
\[
  (-Z_1,-Z_2,-Z_3).
\]
The three scores are independent and identically distributed and have
no ties almost surely.  Hence each prototype wins with probability
$1/3$, so
\[
  Q_d
  \xrightarrow{p}
  \left(\frac13,\frac13,\frac13\right).
\]
Since entropy is continuous and bounded by $\log 3$,
\[
  H_3(P_d,d_{\mathrm{Euc}})
  =
  \mathbb E\{h(Q_d)\}
  \longrightarrow
  h\left(\frac13,\frac13,\frac13\right)
  =
  \log 3.
\]

\item \textbf{Case $\lambda_d\to\lambda\in(0,\infty)$.}

The score reduction and the joint radial convergence
\[
  (\xi_{1,d},\xi_{2,d},\xi_{3,d})
  \Rightarrow
  (\Xi_1,\Xi_2,\Xi_3)
\]
give the limiting conditional score experiment
\[
  \lambda\Xi_i-Z_i,
  \qquad i=1,2,3.
\]
Because the Gaussian noises have continuous distributions, the
limiting scores have no ties almost surely.  The conditional winning
probabilities therefore satisfy
\begin{equation}
  Q_d
  \Rightarrow
  q^\lambda(\Xi_1,\Xi_2,\Xi_3).
  \label{eq:euclidean-q-convergence}
\end{equation}
Entropy is bounded and continuous on the probability simplex, so
\[
  \mathbb E\{h(Q_d)\}
  \longrightarrow
  \mathbb E
  \left[
    h\{q^\lambda(\Xi_1,\Xi_2,\Xi_3)\}
  \right]
  =
  H_\infty(\lambda;F).
\]

\item \textbf{Case $\lambda_d\to\infty$.}

Divide the score representation in \cref{eq:euc-score} by
$\lambda_d$. \Cref{lem:euclidean-score-reduction} gives
\begin{equation}
  \frac{S_{i,d}}{\lambda_d}
  =
  \xi_{i,d}+e_{i,d},
  \qquad
  \max_{1\leq i\leq3}|e_{i,d}|
  \xrightarrow{p}0.
  \label{eq:euclidean-radial-dominance}
\end{equation}

Let $J_d$ be an index attaining
$\max_i\xi_{i,d}$, and let
\[
  \Delta_d
  =
  \xi_{J_d,d}
  -
  \max_{j\neq J_d}\xi_{j,d}
\]
be the gap between the largest and second-largest radial
fluctuations.  Since
\[
  (\xi_{1,d},\xi_{2,d},\xi_{3,d})
  \Rightarrow
  (\Xi_1,\Xi_2,\Xi_3)
\]
and $F$ is continuous, the limiting radial triple has a unique maximum
almost surely.  Consequently,
\[
  \lim_{\eta\downarrow0}
  \limsup_{d\to\infty}
  \mathbb P(\Delta_d\leq2\eta)
  =
  0.
\]

On the event
\[
  \Delta_d>2\eta,
  \qquad
  \max_i|e_{i,d}|\leq\eta,
\]
the perturbations cannot change the radial ordering, and prototype
$J_d$ is the farthest prototype. More explicitly, conditionally on
$V_d$,
\[
  1-Q_{dJ_d}
  \leq
  \mathbb P\!\left(
    \max_i|e_{i,d}|>\eta\mid V_d
  \right)
  +\mathbf 1\{\Delta_d\leq2\eta\}.
\]
The first term converges to zero in probability because its
expectation does, and the second is made arbitrarily small in
probability by first taking $d\to\infty$ and then $\eta\downarrow0$.
It follows that
\[
  \max_{1\leq i\leq3}Q_{di}
  \xrightarrow{p}1.
\]
Thus $Q_d$ approaches a vertex of the probability simplex and
\[
  h(Q_d)\xrightarrow{p}0.
\]
Finally, since $0\leq h(Q_d)\leq\log 3$,
\[
  H_3(P_d,d_{\mathrm{Euc}})
  =
  \mathbb E\{h(Q_d)\}
  \longrightarrow0.
\]
\end{enumerate}
\end{proof}
\section{Hyperbolic curvature amplification}

Let $\kappa_d>0$, let the model space be $\mathbb H^d_{-\kappa_d^2}$, and represent points in geodesic polar coordinates $(R,U)$. The hyperbolic law of cosines gives the increasing distance score \cite{ratcliffe2019}
\begin{equation}
 C_{i,d}=\cosh(\kappa_dR_{0,d})\cosh(\kappa_dR_{i,d})
 -\sinh(\kappa_dR_{0,d})\sinh(\kappa_dR_{i,d})
  \langle U_{0,d},U_{i,d}\rangle.
 \label{eq:hyper-score-exact}
\end{equation}
Let $s_d=\kappa_d\rho_d$ and impose the curvature-scale thinness condition
\begin{equation}
 b_d=\kappa_d\epsilon_d\coth(s_d)\to0.
 \label{eq:hyper-thin}
\end{equation}
Since $b_d=\tau_d A(s_d)$ and $A(s)\ge1$, this also implies relative thinness. Define
\begin{equation}
 \lambda_{d,\mathbb H}=\sqrt d\,b_d
 =\sqrt d\,\epsilon_d\kappa_d\coth(\kappa_d\rho_d)
 =\lambda_dA(s_d).
 \label{eq:lambda-hyper}
\end{equation}

\begin{lemma}[Hyperbolic score reduction]
\label{lem:hyperbolic-score-reduction}
Under \cref{eq:hyper-thin}, the farthest ordering in \cref{eq:hyper-score-exact} is the ordering of
\begin{equation}
 S^{\mathbb H}_{i,d}=\lambda_{d,\mathbb H}\xi_{i,d}-Z_{i,d}+r^{\mathbb H}_{i,d},
 \label{eq:hyper-reduction}
\end{equation}
with the same conditional Gaussian angular limit as in \cref{lem:euclidean-score-reduction}. The remainder $r^{\mathbb H}_{i,d}$ is $o_p(1)$ if $\lambda_{d,\mathbb H}=O(1)$ and $o_p(\lambda_{d,\mathbb H})$ if $\lambda_{d,\mathbb H}\to\infty$.
\end{lemma}

\begin{proof}
Set
\[
  s_d=\kappa_d\rho_d,
  \qquad
  a_{0,d}=\kappa_dR_{0,d},
  \qquad
  u_{i,d}=\kappa_d\epsilon_d\xi_{i,d}.
\]
Then
\[
  \kappa_dR_{i,d}=s_d+u_{i,d},
  \qquad
  a_{0,d}=s_d+u_{0,d}
  =s_d(1+\tau_d\xi_{0,d}).
\]

Subtract the common term
$\cosh(a_{0,d})\cosh(s_d)$ from
\eqref{eq:hyper-score-exact} and divide by the positive random scale
\[
  \frac{\sinh(a_{0,d})\sinh(s_d)}{\sqrt d}.
\]
These operations preserve the ordering of the three prototypes.  The
resulting exact normalized score is
\begin{align}
  \widetilde S^{\mathbb H}_{i,d}
  &=
  \sqrt d\,\coth(a_{0,d})
  \frac{
    \cosh(s_d+u_{i,d})-\cosh(s_d)
  }{\sinh(s_d)}
  \notag\\
  &\qquad
  -
  Z_{i,d}
  \frac{\sinh(s_d+u_{i,d})}{\sinh(s_d)},
  \label{eq:hyperbolic-normalized-exact}
\end{align}
where
\[
  Z_{i,d}
  =
  \sqrt d\,\langle U_{0,d},U_{i,d}\rangle.
\]

The hyperbolic addition formulas give the exact identities
\begin{align}
  \frac{
    \cosh(s_d+u)-\cosh(s_d)
  }{\sinh(s_d)}
  &=
  \sinh u+\coth(s_d)(\cosh u-1),
  \label{eq:hyperbolic-cosh-increment}\\
  \frac{\sinh(s_d+u)}{\sinh(s_d)}
  &=
  \cosh u+\coth(s_d)\sinh u.
  \label{eq:hyperbolic-sinh-ratio}
\end{align}
Since
\[
  b_d
  =
  \kappa_d\epsilon_d\coth(s_d)
  \longrightarrow0
\]
and $\coth(s_d)\geq1$, we have
$\kappa_d\epsilon_d\to0$.  Tightness of $\xi_{i,d}$ therefore gives
$u_{i,d}\to0$ in probability.

Using
\[
  \sinh u
  =
  u+\frac{u^3}{6}+O(u^5e^{|u|}),
  \qquad
  \cosh u
  =
  1+\frac{u^2}{2}+O(u^4e^{|u|}),
\]
equations \eqref{eq:hyperbolic-cosh-increment} and
\eqref{eq:hyperbolic-sinh-ratio} become
\begin{align}
  \frac{
    \cosh(s_d+u_{i,d})-\cosh(s_d)
  }{\sinh(s_d)}
  &=
  u_{i,d}
  +
  \frac{\coth(s_d)}{2}u_{i,d}^2
  \notag\\
  &\qquad
  +
  O_p\!\left(
    |u_{i,d}|^3
    +
    \coth(s_d)|u_{i,d}|^4
  \right),
  \label{eq:hyperbolic-radial-taylor}\\
  \frac{\sinh(s_d+u_{i,d})}{\sinh(s_d)}
  &=
  1+\coth(s_d)u_{i,d}
  +\frac{u_{i,d}^2}{2}
  \notag\\
  &\qquad
  +
  O_p\!\left(
    \coth(s_d)|u_{i,d}|^3
    +
    |u_{i,d}|^4
  \right).
  \label{eq:hyperbolic-angular-taylor}
\end{align}

Substitution into \eqref{eq:hyperbolic-normalized-exact} gives
\begin{align}
  \widetilde S^{\mathbb H}_{i,d}
  &=
  \sqrt d\,\kappa_d\epsilon_d
  \coth(a_{0,d})\xi_{i,d}
  -
  Z_{i,d}
  \notag\\
  &\quad
  +
  \frac{\sqrt d}{2}
  \coth(a_{0,d})\coth(s_d)
  (\kappa_d\epsilon_d)^2\xi_{i,d}^2
  \notag\\
  &\quad
  -
  Z_{i,d}\coth(s_d)
  \kappa_d\epsilon_d\xi_{i,d}
  -
  \frac{Z_{i,d}}{2}
  (\kappa_d\epsilon_d)^2\xi_{i,d}^2
  +
  \widetilde r_{i,d},
  \label{eq:hyperbolic-expanded-score}
\end{align}
where the higher-order remainder satisfies
\[
  \max_i|\widetilde r_{i,d}|=o_p(1)
  \quad\text{if }\lambda_{d,\mathbb H}=O(1),
\]
and
\[
  \frac{\max_i|\widetilde r_{i,d}|}
       {\lambda_{d,\mathbb H}}
  \xrightarrow{p}0
  \quad\text{if }\lambda_{d,\mathbb H}\to\infty.
\]

The leading radial term in
\eqref{eq:hyperbolic-expanded-score} contains
$\coth(a_{0,d})=\coth(\kappa_dR_{0,d})$.  This has a direct geometric
interpretation: in hyperbolic space, the change in distance produced by
moving a prototype radially depends on the radial position of the
query.  Thus the query radius enters the exact first-order sensitivity.
The thin-shell condition then shows that the random query radius is
asymptotically indistinguishable from the common shell radius; it is
not replaced by $\rho_d$ as an additional modeling assumption.

Indeed,
\[
  a_{0,d}
  =
  s_d(1+\tau_d\xi_{0,d}),
  \qquad
  \tau_d
  =
  \frac{b_d}{s_d\coth(s_d)}
  \longrightarrow0.
\]
Consequently,
\begin{equation}
  \frac{\coth(a_{0,d})}{\coth(s_d)}
  \xrightarrow{p}1.
  \label{eq:hyperbolic-coth-query-replacement}
\end{equation}
For $s_d\to0$, this follows from
$\coth t\sim t^{-1}$ and
$a_{0,d}/s_d\to1$; when $s_d$ is bounded away from zero, it follows
from $a_{0,d}-s_d=u_{0,d}\to0$ and continuity of $\coth$; and when
$s_d\to\infty$, both hyperbolic cotangents converge to one.

It follows from \eqref{eq:hyperbolic-coth-query-replacement} that
\begin{align}
  \sqrt d\,\kappa_d\epsilon_d
  \coth(a_{0,d})\xi_{i,d}
  &=
  \sqrt d\,\kappa_d\epsilon_d
  \coth(s_d)\xi_{i,d}
  +
  o_p(\lambda_{d,\mathbb H})
  \notag\\
  &=
  \lambda_{d,\mathbb H}\xi_{i,d}
  +
  o_p(\lambda_{d,\mathbb H}).
  \label{eq:hyperbolic-leading-radial-term}
\end{align}
When $\lambda_{d,\mathbb H}=O(1)$, the error in
\eqref{eq:hyperbolic-leading-radial-term} is $o_p(1)$.

The explicit correction terms in
\eqref{eq:hyperbolic-expanded-score} are also negligible.  In
particular,
\[
  \coth(s_d)\kappa_d\epsilon_d\xi_{i,d}
  =
  b_d\xi_{i,d}
  =
  o_p(1),
\]
so the angular coefficient converges to one.  Moreover, the
second-order radial term, divided by $\lambda_{d,\mathbb H}$, is
\[
  \frac12
  \kappa_d\epsilon_d
  \coth(a_{0,d})\xi_{i,d}^2
  =
  O_p(b_d)
  =
  o_p(1).
\]
The remaining displayed terms are of still smaller order.

Combining these estimates yields
\[
  \widetilde S^{\mathbb H}_{i,d}
  =
  \lambda_{d,\mathbb H}\xi_{i,d}
  -
  Z_{i,d}
  +
  r^{\mathbb H}_{i,d},
\]
where
\[
  \max_i|r^{\mathbb H}_{i,d}|
  \xrightarrow{p}0
  \quad\text{if }\lambda_{d,\mathbb H}=O(1),
\]
and
\[
  \frac{\max_i|r^{\mathbb H}_{i,d}|}
       {\lambda_{d,\mathbb H}}
  \xrightarrow{p}0
  \quad\text{if }\lambda_{d,\mathbb H}\to\infty.
\]
Finally, the conditional spherical Gaussian limit for
$(Z_{1,d},Z_{2,d},Z_{3,d})$ is the same as in
  \cref{lem:euclidean-score-reduction}. Since
$\widetilde S^{\mathbb H}_{i,d}$ has the same ordering as the exact
hyperbolic distance score, the result follows.
\end{proof}

\paragraph{Remark.} The condition $b_d\to0$ has a direct geometric interpretation.  In
geodesic polar coordinates, the hyperbolic metric is
\[
  \mathrm ds^2
  =
  \mathrm dr^2
  +
  s_{\kappa_d}(r)^2\,\mathrm d\Omega^2,
  \qquad
  s_{\kappa_d}(r)
  =
  \frac{\sinh(\kappa_dr)}{\kappa_d}.
\]
The function $s_{\kappa_d}(r)$ is the tangential length scale at radius
$r$, and
\[
  \frac{\mathrm d}{\mathrm dr}
  \log s_{\kappa_d}(r)
  =
  \kappa_d\coth(\kappa_dr).
\]
Consequently,
\[
  b_d
  =
  \epsilon_d
  \left.
  \frac{\mathrm d}{\mathrm dr}
  \log s_{\kappa_d}(r)
  \right|_{r=\rho_d}
\]
measures the fractional change in the local angular scale across one
radial standard deviation.  Equivalently,
\[
  b_d
  =
  \frac{\epsilon_d}
  {\ell_{\mathrm{geom}}(\rho_d)},
  \qquad
  \ell_{\mathrm{geom}}(\rho_d)
  =
  \frac{\tanh(\kappa_d\rho_d)}{\kappa_d}.
\]
Thus $b_d\to0$ requires the radial shell to be narrow relative to the
local geometric variation length.  It reduces to
$\epsilon_d/\rho_d\to0$ near the flat regime and to
$\kappa_d\epsilon_d\to0$ deep in the hyperbolic regime.  This is a
local linearization condition, not a negligible-curvature condition:
the amplified parameter $\lambda_{d,\mathbb H}=\sqrt d\,b_d$ may still have a
nonzero or divergent limit.

\begin{theorem}[Hyperbolic shell trichotomy]
\label{thm:hyperbolic-shell-trichotomy}
Under \cref{eq:hyper-thin}, replace $\lambda_d$ in
\cref{thm:euclidean-shell-trichotomy} by $\lambda_{d,\mathbb H}$.
The three entropy limits are $\log 3$, $\Hinf(\lambda;F)$, and $0$ in
the angular-dominated, intermediate, and radial-dominated regimes,
respectively; the radial-dominated conclusion additionally requires
the limiting law $F$ to be continuous.
\end{theorem}

\begin{proof}
Apply the proof of \cref{thm:euclidean-shell-trichotomy} to
\cref{eq:hyper-reduction}.
\end{proof}

The amplification factor has the expansions
\begin{align}
 A(s)&=1+\frac{s^2}{3}-\frac{s^4}{45}+O(s^6),&&s\to0,\label{eq:A-small}\\
 A(s)&=s\{1+2e^{-2s}+O(e^{-4s})\},&&s\to\infty.\label{eq:A-large}
\end{align}
Thus curvature sensitivity is quadratic near the flat limit and asymptotically linear deep in the hyperbolic regime.

\begin{proposition}[Monotonicity of geometric amplification]
\label{prop:A-monotone}
With the continuous extension $A(0)=1$, the map $A(s)=s\coth s$ is strictly increasing from $[0,\infty)$ onto $[1,\infty)$.
\end{proposition}

\begin{proof}
For $s>0$,
\begin{equation}
 A'(s)=\coth s-s\operatorname{csch}^2s
 =\frac{\sinh s\cosh s-s}{\sinh^2s}>0.
\end{equation}
The numerator is zero at $s=0$ and has derivative $2\sinh^2s>0$ for $s>0$. The endpoint values follow from the expansions above.
\end{proof}

\subsection{Dimensionless inversion and the ordinal-embedding interface}

Write
\begin{equation}
 \tau_d=\frac{\epsilon_d}{\rho_d},\qquad
 s_d=\kappa_d\rho_d,\qquad
 \chi_d=s_d^2=\kappa_d^2\rho_d^2.
 \label{eq:dimensionless-parameters}
\end{equation}
Then \cref{eq:lambda-hyper} becomes
\begin{equation}
 \lambda_{d,\mathbb H}=\sqrt d\,\tau_d A(s_d).
 \label{eq:dimensionless-inversion}
\end{equation}

\begin{proposition}[Similarity-gauge invariance]
\label{prop:gauge}
Under a change of length unit $\metric'=c\metric$ with $c>0$, let
$\rho_d'=c\rho_d$, $\epsilon_d'=c\epsilon_d$, and $\kappa_d'=\kappa_d/c$. Then
\begin{equation}
 \Htri(P,\metric')=\Htri(P,\metric),\qquad
 \tau_d'=\tau_d,\qquad s_d'=s_d,\qquad\chi_d'=\chi_d,
\end{equation}
and $\lambda_{d,\mathbb H}'=\lambda_{d,\mathbb H}$.
\end{proposition}

\begin{proof}
The entropy identity follows from \cref{prop:ordinal}. The remaining identities follow by direct substitution:
$\epsilon_d'/\rho_d'=\epsilon_d/\rho_d$ and
$\kappa_d'\rho_d'=(\kappa_d/c)(c\rho_d)=\kappa_d\rho_d$.
\end{proof}

Suppose $\Hinf(\lambda;F)$ is strictly decreasing on a chosen operating interval and the radial law, $d$, and $\tau_d$ are known or consistently calibrated. An entropy observation identifies
\begin{equation}
 \widehat\lambda_{d,\mathbb H}=\Hinf^{-1}(\widehat H;F),
\end{equation}
after which \cref{prop:A-monotone} gives the unique normalized solution
\begin{equation}
 \widehat s_d=A^{-1}\!\left(
 \frac{\widehat\lambda_{d,\mathbb H}}{\sqrt d\,\widehat\tau_d}
 \right),\qquad
 \widehat\chi_d=\widehat s_d^2,
 \label{eq:chi-hat}
\end{equation}
provided the argument of $A^{-1}$ lies in $[1,\infty)$. Absolute curvature can then be reported as $\widehat\kappa_d^2=\widehat\chi_d/\rho_d^2$ only when $\rho_d$ is expressed in an externally meaningful length unit.

An ordinal embedding is not needed to estimate $\Htri$, but it is one possible source of the quantities required for curvature calibration. If a particular embedding method is known to recover the relevant shell up to a similarity, any margin-induced global scale changes $\rho_d$ and $\epsilon_d$ together and leaves $\widehat\tau_d$ and $\widehat\chi_d$ unchanged. The margin fixes the numerical coordinate unit; it does not identify physical scale. Sparse comparisons, noisy labels, or model misspecification may distort an embedding by more than a similarity, so this interface requires a separate embedding-error analysis \cite{kleindessner2014,ariascastro2017,suzuki2019}.

The remaining qualifications are structural. The flat expansion \cref{eq:A-small} makes small normalized curvature weakly identified. The master curve flattens near zero entropy, making large-curvature inversion ill-conditioned. Finally, isotropy and radial-angular independence are model assumptions, not conclusions drawn from the ordinal observations.

For a non-Gaussian shell, one can standardize observed radii and sample $\Xi$ triples from their empirical distribution when evaluating \cref{eq:q-integral}. This corrects radial-law mismatch but does not correct angular anisotropy or radial-angular dependence.

\section{Experiments}

\subsection{Protocol}

All main results are regenerated by the accompanying CPU notebook. It first checks that farthest labels are unchanged by several strictly increasing transforms and that $\tau_d$, $s_d$, $\chi_d$, and $\lambda_{d,\mathbb H}$ remain invariant under several global rescalings. These checks verify implementations of exact identities and are not empirical evidence for the theorems. Unless noted otherwise, each simulation condition uses $M=220$ prototype triples, $N=700$ witnesses per triple, the adaptive correction in \cref{eq:miller}, and independent fixed seeds. The master curve uses $18{,}000$ radial triples and $35$-node Gauss--Hermite quadrature. Error bars in the figures are $\pm1.96$ Monte Carlo standard errors and do not include data-sampling uncertainty. We report continuous error summaries rather than applying a post hoc pass--fail threshold.

\subsection{Euclidean master curve}

We generated Gaussian radial shells at $d\in\{32,96,256\}$ and $\lambda\in\{0,0.5,1,2,3\}$. Positive radii were enforced by a floor at $10^{-3}$; the maximum affected fraction was $3.38\%$, attained at the lowest dimension and largest shell width. Across all 15 conditions, the RMSE from the Gaussian master curve was $0.0164$ and the maximum absolute residual was $0.0276$ (\cref{fig:euclidean}). The residual does not decrease monotonically in this modest Monte Carlo run, so the figure supports finite-dimensional approximation rather than a claimed empirical convergence rate.

\begin{figure}[t]
  \centering
  \includegraphics[width=\textwidth]{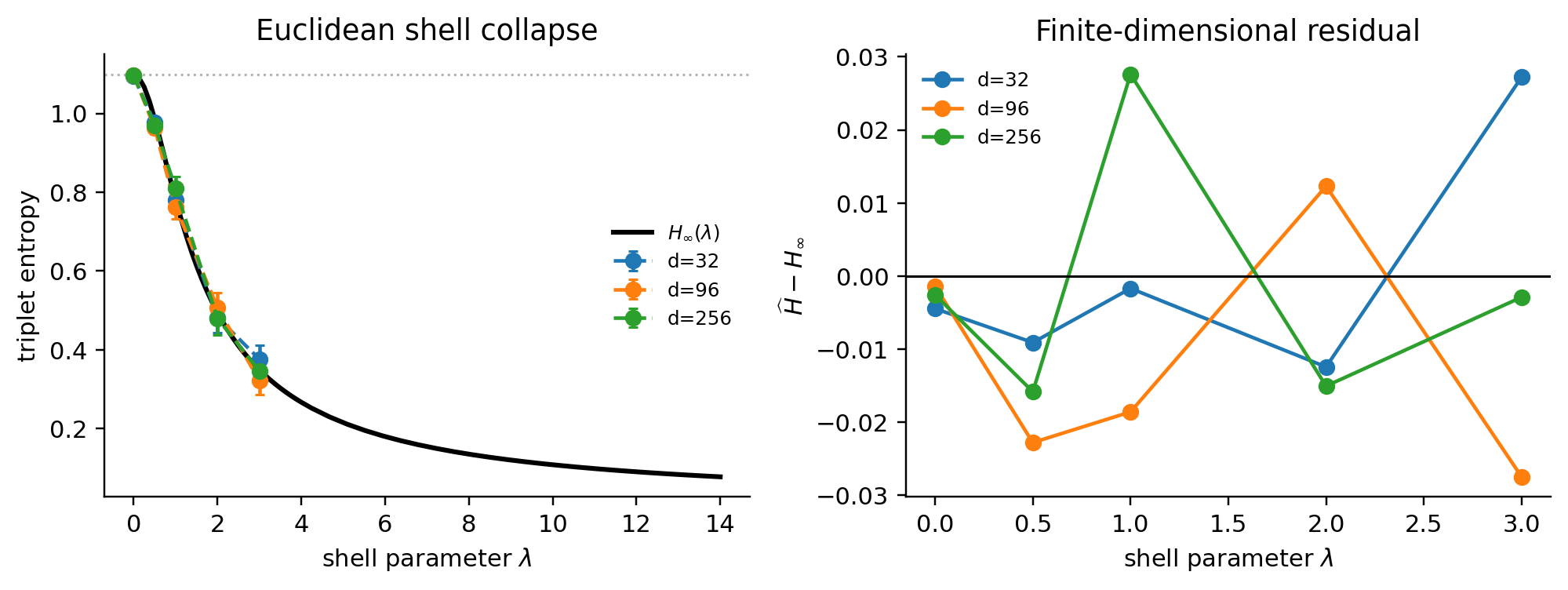}
  \caption{\textbf{Euclidean shell law.} Left: the finite-dimensional estimates track the limiting curve as a function of $\lambda$. Right: signed residuals show finite-dimensional and Monte Carlo deviations; no pointwise equality is claimed.}
  \label{fig:euclidean}
\end{figure}

\subsection{Hyperbolic amplification}

At $d=128$, $\rho=1$, and $\lambda_d=1$, we varied $\kappa$ from $0.05$ to $5$ and evaluated the exact hyperbolic law-of-cosines ordering. Plotting entropy at $\lambda_{d,\mathbb H}=A(\kappa)$ gives RMSE $0.0209$ and maximum absolute error $0.0389$ from the same Gaussian master curve (\cref{fig:hyperbolic}). The Spearman correlation between effective parameter and entropy was $-1$ in this run. The right panel illustrates the quadratic flat-regime approximation and the deep-hyperbolic linear approximation.

\begin{figure}[t]
  \centering
  \includegraphics[width=\textwidth]{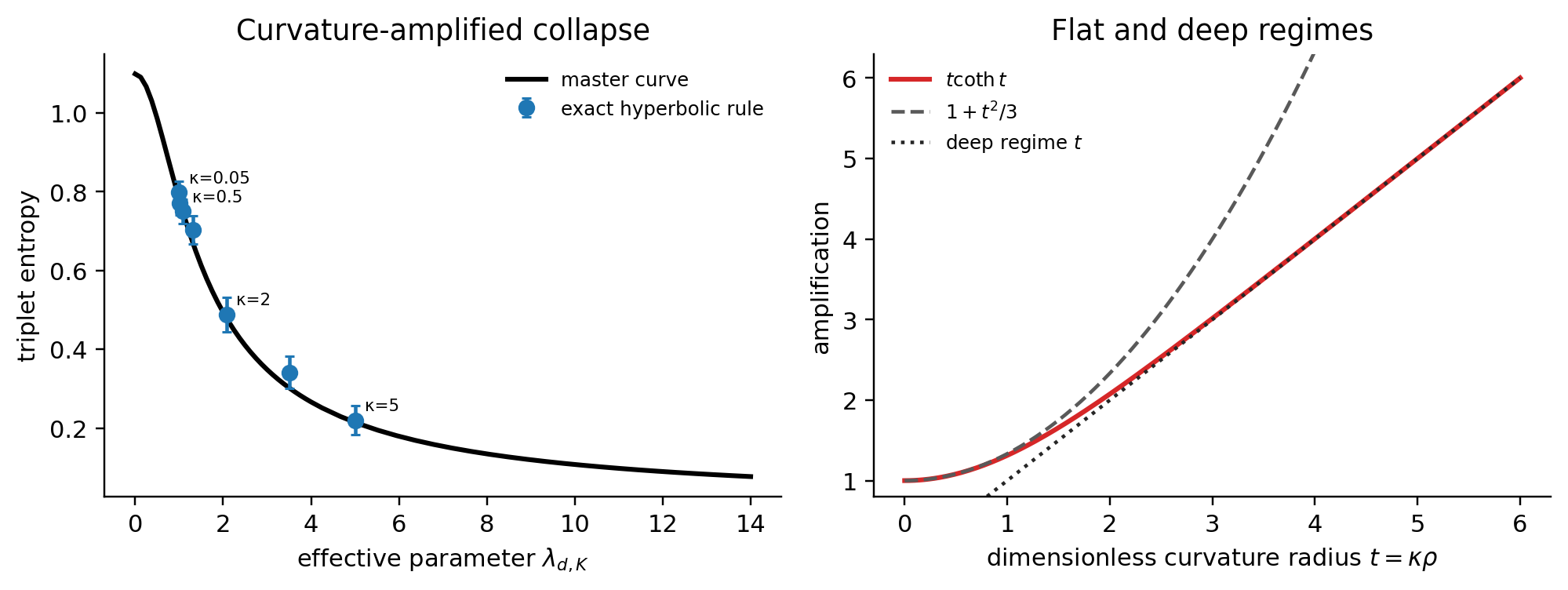}
  \caption{\textbf{Hyperbolic amplification.} Left: exact finite-dimensional hyperbolic estimates plotted at the predicted effective parameter. Right: $A(s)=s\coth s$ and its flat- and deep-regime asymptotic forms.}
  \label{fig:hyperbolic}
\end{figure}

\subsection{Controlled latent-network calibration}

We sampled $700$ synthetic latent points in $d=48$, generated sparse hyperbolic graphs with average degree approximately $10$, selected an annulus of $564$ observed latent points, calibrated its empirical radial law, and inverted entropy at four values of the curvature scale $\kappa$ \cite{krioukov2010,papadopoulos2012,lixuzhu2026}. The estimator used the latent polar coordinates; adjacency was used only to verify a nontrivial sparse-network regime. The median relative error in $\kappa$ was $6.6\%$ and the maximum was $12.5\%$ (\cref{tab:summary,fig:stress}, left). This is a controlled calibration result, not evidence for graph-only identification.

\begin{table}[t]
\centering
\caption{Compact reproduction summary. ``Worst case'' is an absolute entropy error for the master-curve experiments and a relative error in the curvature scale $\kappa$ for the inversion experiments.}
\label{tab:summary}
\begin{tabular}{lrrl}
\toprule
Experiment & Summary statistic & Worst case & Scope \\
\midrule
Euclidean shell & RMSE $0.0164$ & $0.0276$ & $15$ conditions \\
Hyperbolic amplification & RMSE $0.0209$ & $0.0389$ & $7$ curvatures \\
Latent-network calibration & median $6.6\%$ & $12.5\%$ & observed synthetic latents \\
Angular anisotropy & Spearman $1.00$ & $68.9\%$ & stress test \\
\bottomrule
\end{tabular}
\end{table}

\subsection{Angular-anisotropy stress test}

At true $\kappa=2.2$, we progressively amplified eight angular coordinates and then renormalized directions, while continuing to use the isotropic empirical-radial calibration. The angular covariance eigenvalue coefficient of variation rose from $0.286$ to $1.872$, and the relative error in $\kappa$ rose monotonically from $1.0\%$ to $68.9\%$ (Spearman $1.00$; \cref{fig:stress}, right). This stress test shows that accurate radial-law calibration does not compensate for misspecification of the angular model.

\begin{figure}[t]
  \centering
  \includegraphics[width=\textwidth]{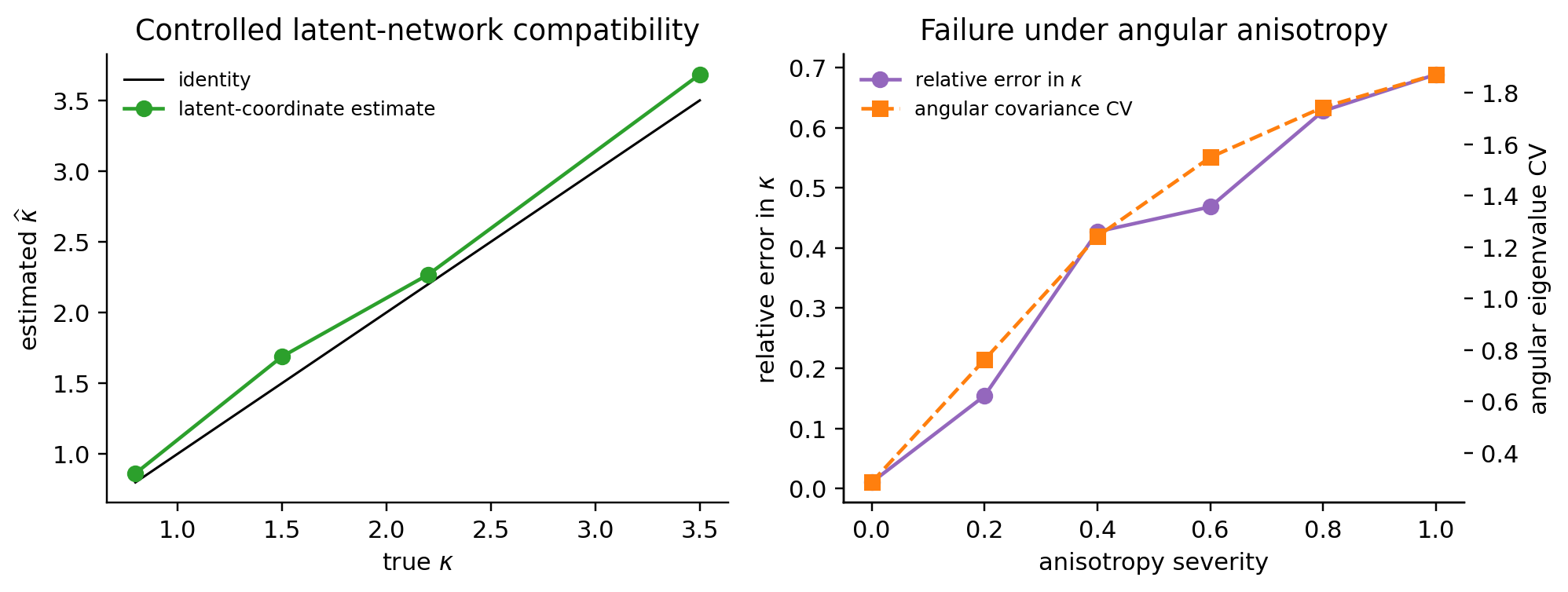}
  \caption{\textbf{Calibration scope and failure mode.} Left: inversion from observed latent coordinates in controlled synthetic hyperbolic networks. Right: isotropic calibration fails progressively as angular covariance becomes anisotropic.}
  \label{fig:stress}
\end{figure}

\section{Discussion and limitations}

Farthest-cell triplet entropy is a rank-based conditional entropy statistic. The score reductions explain why the same master curve appears in Euclidean and hyperbolic shells, and why curvature enters only through the scale-invariant pair $(\tau,\kappa\rho)$. The information identity \cref{eq:mi} clarifies that entropy defect measures how much the farthest label reveals about prototype radial fluctuations. The statistic can be computed without an embedding, but curvature calibration cannot be separated from assumptions about the shell law.

Several gaps remain before the method could support broad applied claims.
\begin{enumerate}
  \item \textbf{No general monotonicity theorem.} Strict decrease of $\Hinf$ is numerically observed for the Gaussian and empirical laws used here, but inversion for an arbitrary radial law must verify its operating interval.
  \item \textbf{No finite-dimensional rate.} The asymptotic proofs give convergence but not a Berry-Esseen-type bound for conditional cell probabilities or entropy.
  \item \textbf{Strong model assumptions.} Isotropy, radial-angular independence, and curvature-scale thinness are essential for the master curve. The stress experiment shows that fixed angular anisotropy can produce large inversion bias.
  \item \textbf{Scale gauge.} Under the calibrated shell model and a verified monotone operating interval, entropy inversion targets $\chi=\kappa^2\rho^2$. Reporting absolute curvature requires an external length unit. An ordinal-embedding margin supplies a reproducible convention, not physical-scale information.
  \item \textbf{Embedding interface.} Comparison-only estimation of $\Htri$ does not require coordinates. Using an ordinal embedding to estimate $d$, $\tau$, or a radial law introduces embedding error that is not covered by the shell theorem.
  \item \textbf{Uncertainty.} The reported standard errors measure Monte Carlo variation. Population inference from a learned finite embedding requires resampling or a data-generating analysis that accounts for embedding uncertainty.
  \item \textbf{Synthetic application only.} The present study does not establish graph-only recovery or provide a real-data curvature estimate. Inference from learned coordinates requires a separate analysis of embedding error and metric-scale calibration.
  \item \textbf{More than three prototypes.} The limiting construction and ceiling extend mechanically to any fixed $m>3$, with ceiling $\log m$. Explicit cell-probability formulas, finite-dimensional reductions, and the associated statistical tradeoffs are not analyzed here.
\end{enumerate}

A natural next step is a nonasymptotic theory under quantified angular anisotropy, together with a statistical analysis of incomplete or noisy comparison oracles. Until such a correction is established, normalized-curvature inversion should be reported as calibrated, model-based inference accompanied by diagnostic checks and anisotropy stress tests.

\paragraph{Invitation for collaboration.}
The author welcomes correspondence and potential collaboration with researchers interested in developing the open directions identified above. In particular, opportunities for substantial contributions include finite-sample concentration bounds, inference under incomplete or noisy triplet comparisons, robustness corrections for angular anisotropy, and empirical validation on real data. Interested researchers are invited to contact the author at \href{mailto:cd3411@columbia.edu}{cd3411@columbia.edu}.

\section*{Reproducibility statement}

The accompanying notebook contains the estimator, quadrature formula, all random seeds and budgets, raw tables, figure generation, and numerical consistency checks. It runs on a CPU with NumPy, SciPy, pandas, and Matplotlib. Every number in \cref{tab:summary} is generated by the notebook.

\appendix
\section{Conditional spherical Gaussian step}

For completeness, fix three prototype directions and form the $d\times3$ matrix $B_d=[U_{1,d},U_{2,d},U_{3,d}]$. If $g_d\sim N(0,I_d)$, then $U_{0,d}=g_d/\lVert g_d\rVert$ is uniform on $S^{d-1}$ and
\begin{equation}
 \sqrt d\,B_d^\top U_{0,d}
 =B_d^\top g_d\,\frac{\sqrt d}{\lVert g_d\rVert}.
\end{equation}
Conditionally on $B_d$, $B_d^\top g_d$ is Gaussian with covariance $B_d^\top B_d$. The scalar factor converges in probability to one, while the random Gram matrix converges in probability to $I_3$ because independent spherical directions have pairwise inner products tending to zero. This proves the conditional-in-probability Gaussian approximation used in both score reductions. Since the limit has a density, argmax boundaries have probability zero.

\section{Admissible scaling for the large-shell regime}

The condition $\epsilon_d/\rho_d=O(d^{-1/2})$ restricts $\lambda_d=\sqrt d\,\epsilon_d/\rho_d$ to a bounded regime. The more general assumption $\epsilon_d/\rho_d\to0$ controls the Euclidean Taylor remainders while allowing $\lambda_d$ to diverge. Hyperbolically, the analogous quantity is $b_d=\kappa_d\epsilon_d\coth(\kappa_d\rho_d)$; requiring $b_d\to0$ controls both the angular coefficient and the second-order radial term while allowing $\sqrt d\,b_d$ to diverge.

\section{Numerical inversion details}

The notebook evaluates \cref{eq:q-integral} with Gauss-Hermite quadrature and averages over sampled radial triples. To protect inversion from Monte Carlo upward steps, it replaces a computed sequence $H_1,\ldots,H_G$ on an increasing $\lambda$ grid by the cumulative minimum $\widetilde H_g=\min_{j\le g}H_j$. This is a numerical monotone projection, not evidence for a general monotonicity theorem. Entropy values outside the calibrated range are clipped to the grid endpoints and should be reported as saturation rather than precise curvature estimates. Likewise, if sampling error produces $\widehat\lambda/(\sqrt d\,\widehat\tau)<1$, the dimensionless inversion is reported at the flat boundary $\widehat\chi=0$ rather than extrapolated outside the range of $A$.


\begin{thebibliography}{99}

\bibitem{shannon1948}
C. E. Shannon.
\newblock A mathematical theory of communication.
\newblock \emph{Bell System Technical Journal},
27(3):379--423 and 27(4):623--656, 1948.
\newblock \doi{10.1002/j.1538-7305.1948.tb01338.x};
\doi{10.1002/j.1538-7305.1948.tb00917.x}.

\bibitem{miller1955}
G. A. Miller.
\newblock Note on the bias of information estimates.
\newblock In H. Quastler, editor,
\emph{Information Theory in Psychology: Problems and Methods},
volume II-B, pages 95--100. Free Press, Glencoe, IL, 1955.

\bibitem{paninski2003}
L. Paninski.
\newblock Estimation of entropy and mutual information.
\newblock \emph{Neural Computation}, 15(6):1191--1253, 2003.
\newblock \doi{10.1162/089976603321780272}.

\bibitem{aurenhammer1991}
F. Aurenhammer.
\newblock Voronoi diagrams: a survey of a fundamental geometric data structure.
\newblock \emph{ACM Computing Surveys}, 23(3):345--405, 1991.
\newblock \doi{10.1145/116873.116880}.

\bibitem{bormashenko2024}
E. Bormashenko, S. Shoval, M. Frenkel, and M. Nosonovsky.
\newblock Voronoi entropy and long-range order of 2D point sets.
\newblock arXiv:2410.21668, 2024.
\newblock \url{https://arxiv.org/abs/2410.21668}.

\bibitem{diaconisfreedman1987}
P. Diaconis and D. Freedman.
\newblock A dozen de Finetti-style results in search of a theory.
\newblock \emph{Annales de l'Institut Henri Poincar\'e,
Probabilit\'es et Statistiques}, 23(S2):397--423, 1987.
\newblock \url{https://www.numdam.org/item/AIHPB_1987__23_S2_397_0/}.

\bibitem{heiny2022}
J. Heiny, S. Johnston, and J. Prochno.
\newblock Thin-shell theory for rotationally invariant random simplices.
\newblock \emph{Electronic Journal of Probability},
27, article no.~2, 1--41, 2022.
\newblock \doi{10.1214/21-EJP734}.

\bibitem{kleindessner2014}
M. Kleindessner and U. von Luxburg.
\newblock Uniqueness of ordinal embedding.
\newblock In \emph{Proceedings of the 27th Conference on Learning Theory}, volume 35 of \emph{Proceedings of Machine Learning Research}, pages 40--67, 2014.
\newblock \url{https://proceedings.mlr.press/v35/kleindessner14.html}.

\bibitem{ariascastro2017}
E. Arias-Castro.
\newblock Some theory for ordinal embedding.
\newblock \emph{Bernoulli}, 23(3):1663--1693, 2017.
\newblock \doi{10.3150/15-BEJ792}.

\bibitem{suzuki2019}
A. Suzuki, J. Wang, F. Tian, A. Nitanda, and K. Yamanishi.
\newblock Hyperbolic ordinal embedding.
\newblock In \emph{Proceedings of the Eleventh Asian Conference on Machine Learning}, volume 101 of \emph{Proceedings of Machine Learning Research}, pages 1065--1080, 2019.
\newblock \url{https://proceedings.mlr.press/v101/suzuki19a.html}.

\bibitem{ratcliffe2019}
J. G. Ratcliffe.
\newblock \emph{Foundations of Hyperbolic Manifolds}.
\newblock Graduate Texts in Mathematics, volume 149. Springer, Cham, third edition, 2019.
\newblock \doi{10.1007/978-3-030-31597-9}.

\bibitem{krioukov2010}
D. Krioukov, F. Papadopoulos, M. Kitsak, A. Vahdat, and M. Bogu\~n\'a.
\newblock Hyperbolic geometry of complex networks.
\newblock \emph{Physical Review E}, 82:036106, 2010.
\newblock \doi{10.1103/PhysRevE.82.036106}.

\bibitem{papadopoulos2012}
F. Papadopoulos, M. Kitsak, M. A. Serrano, M. Bogu\~n\'a, and D. Krioukov.
\newblock Popularity versus similarity in growing networks.
\newblock \emph{Nature}, 489:537--540, 2012.
\newblock \doi{10.1038/nature11459}.

\bibitem{nickelkiela2017}
M. Nickel and D. Kiela.
\newblock Poincar\'e embeddings for learning hierarchical representations.
\newblock In \emph{Advances in Neural Information Processing Systems 30}, 2017.
\newblock \url{https://proceedings.neurips.cc/paper/2017/hash/59dfa2df42d9e3d41f5b02bfc32229dd-Abstract.html}.

\bibitem{gu2019}
A. Gu, F. Sala, B. Gunel, and C. R\'e.
\newblock Learning mixed-curvature representations in product spaces.
\newblock In \emph{International Conference on Learning Representations}, 2019.
\newblock \url{https://openreview.net/forum?id=HJxeWnCcF7}.

\bibitem{skopek2020}
O. Skopek, O.-E. Ganea, and G. B\'ecigneul.
\newblock Mixed-curvature variational autoencoders.
\newblock In \emph{International Conference on Learning Representations}, 2020.
\newblock \url{https://openreview.net/forum?id=S1g6xeSKDS}.

\bibitem{bachmann2020}
G. Bachmann, G. B\'ecigneul, and O.-E. Ganea.
\newblock Constant curvature graph convolutional networks.
\newblock In \emph{Proceedings of the 37th International Conference on Machine Learning}, volume 119 of \emph{Proceedings of Machine Learning Research}, pages 486--496, 2020.
\newblock \url{https://proceedings.mlr.press/v119/bachmann20a.html}.

\bibitem{lixuzhu2026}
J. Li, G. Xu, and J. Zhu.
\newblock Hyperbolic network latent space model with learnable curvature.
\newblock arXiv:2312.05319v2, 2026.
\newblock \url{https://arxiv.org/abs/2312.05319v2}.

\bibitem{hausmanwise1978}
J. A. Hausman and D. A. Wise.
\newblock A conditional probit model for qualitative choice: Discrete
decisions recognizing interdependence and heterogeneous preferences.
\newblock \emph{Econometrica}, 46(2):403--426, 1978.
\newblock \doi{10.2307/1913909}.

\end{thebibliography}
\end{document}